\pdfoutput=1
\documentclass{lms}
\usepackage{amssymb}
\usepackage{tikz-cd}
\usepackage{mathtools}
\usepackage{hyperref}
\usepackage[T1]{fontenc}
\usepackage{medstar}
\usepackage{microtype}

\newtheorem{theorem}[subsection]{Theorem}
\newtheorem{corollary}[subsection]{Corollary}
\newtheorem{proposition}[subsection]{Proposition}
\newtheorem{lemma}[subsection]{Lemma}
\newnumbered{example}[subsection]{Example}
\newnumbered{definition}[subsection]{Definition}
\newnumbered{remark}[subsection]{Remark}
\newnumbered{exercise}[subsection]{Exercise}

\DeclareMathOperator{\aut}{Aut}
\DeclareMathOperator{\fun}{Fun}
\DeclareMathOperator*{\colim}{colim}
\DeclareMathOperator{\mor}{Hom}
\DeclareMathOperator{\ob}{ob}
\DeclareMathOperator{\morphisms}{mor}
\DeclareMathOperator{\sgn}{sgn}
\DeclareMathOperator{\id}{id}

\newcommand{\op}{\mathrm{op}}
\newcommand{\simpcat}{\mathbf{sCat}}
\newcommand{\simpcatinv}{\mathbf{isCat}}
\newcommand{\sset}{\mathbf{sSet}}
\newcommand{\groundcat}{\mathcal{K}}
\newcommand{\ssetinv}{\mathbf{isSet}}
\newcommand{\hcr}{\mathfrak{C}}
\newcommand{\hcn}{N_{hc}}
\newcommand{\hcrinv}{\widetilde{\mathfrak{C}}}
\newcommand{\hcninv}{\widetilde{N}_{hc}}
\newcommand{\catinv}{\mathbf{iCat}}
\newcommand{\ground}{\mathcal{K}}
\newcommand{\cat}{\mathbf{Cat}}
\newcommand{\dcat}{\mathbf{dCat}}
\newcommand{\set}{\mathbf{Set}}
\newcommand{\gpd}{\mathbf{Gpd}}
\newcommand{\rlp}[1]{{#1}^\boxslash}
\newcommand{\llp}[1]{{}^\boxslash\!{#1}}

\title[A criterion for right-induced model structures]{A criterion for existence of right-induced model structures}

\author[G.~C.~Drummond-Cole and P.~Hackney]{Gabriel C. Drummond-Cole and Philip Hackney}
\extraline{The first author was supported by IBS-R003-D1.
The second author acknowledges the support of Australian Research Council Discovery Project grant DP160101519.}

\classno{Primary: 55U35, 18A40. Secondary: 18F20, 55U10, 18D50}

\begin{document}

\maketitle
\begin{abstract}
Suppose that $F: \mathcal{N} \to \mathcal{M}$ is a functor whose target is a Quillen model category. 
We give a succinct sufficient condition for the existence of the right-induced model category structure on $\mathcal{N}$ in the case when $F$ admits both adjoints.
We give several examples, including change-of-rings, operad-like structures, and anti-involutive structures on infinity categories.
For the last of these, we explore anti-involutive structures for several different models of $(\infty,1)$-categories, and show that known Quillen equivalences between base model categories lift to equivalences.
\end{abstract}
\maketitle
\section{Introduction}
Suppose that $\mathcal{M}$ is a model category and $F: \mathcal{N} \to \mathcal{M}$ is a functor. 
A standard question is whether we can use $F$ to build a model structure on $\mathcal{N}$ with the property that a morphism $f$ is a weak equivalence in $\mathcal{N}$ if and only if $F(f)$ is a weak equivalence in $\mathcal{M}$, that is, so that $F$ `creates' weak equivalences.
When $F$ is a right adjoint and $\mathcal{M}$ is cofibrantly generated, there is a classical theorem indicating that in good cases we can do this $F$ also creates fibrations. 
We here give a simple criterion (Theorem~\ref{theorem: adjoint string, Quillen condition}) for the existence of this type of `right-induced' model structure: if $F$ has both a left adjoint $L$ and a right adjoint $R$ so that $(FL, FR)$ is a Quillen self-adjunction on $\mathcal{M}$, then such a model structure exists on $\mathcal{N}$.

We apply this criterion to several interesting examples in this paper, some known, some folkloric, and some new. 
A primary source of examples come from diagram categories: if $\mathcal{C} \to \mathcal{D}$ is a functor and $\groundcat$ is bicomplete, then $\groundcat^{\mathcal{D}} \to \groundcat^{\mathcal{C}}$ admits both adjoints, given by Kan extension. 
If a category $\mathcal{C}$ admits an action by a group $G$, then the inclusion of $\mathcal{C}$ into the semi-direct product $\mathcal{C} \rtimes G$ is an example of such a functor; in Theorem~\ref{theorem: semidirect} we explain suitable conditions for the group action to be compatible with a model structure on $\groundcat^{\mathcal{C}}$.

One salient example is the nontrivial action of $C_2$ on the simplicial category $\Delta$; presheaves on $\nabla \coloneqq \Delta \rtimes C_2$ are called `real simplicial sets' in the literature on real algebraic K-theory (see, for instance,~\cite{Dotto:ECFBZ2ARAK,HesselholtMadsen:RAKT}).
We show here how the Joyal model structure on the category of simplicial sets, which models $(\infty,1)$-categories, lifts to a model structure on $\nabla$-presheaves.

We also study categories with anti-involution; particular examples include exact categories with strict duality~\cite{Schlichting:HKTEC} and dagger categories~\cite{Selinger:DCCCCPM,Baez:QQCTP}.
We show that the folk model structure on $\cat$ lifts to a model structure on categories with anti-involution.
Similarly, the Bergner model structure on simplicial categories lifts to a model structure on simplicial categories with anti-involution. Simplicial categories with the Bergner structure also model $(\infty,1)$-categories, 
and so we have two model categories which should model $(\infty,1)$-categories with (strict) anti-involution.

Homotopy coherent nerve and realization form a Quillen equivalence between simplicial sets with the Joyal model structure and simplicial categories with the Bergner model structure (see, e.g.,~\cite{Lurie:HTT,DuggerSpivak:MSQC}).  This adjunction lifts to the anti-involutive versions of these categories, and we show that the lift is also a Quillen equivalence (Theorem~\ref{theorem: hcn qe lifts}). 
The same technique can be used to show that Quillen equivalences lift to Quillen equivalences on the model structures guaranteed by Theorem~\ref{theorem: adjoint string, Quillen condition}, as long as the adjunctions themselves lift (Theorem~\ref{theorem lanari}).
In particular, we expect that our methods should yield a straightforward argument that various models for $(\infty,n)$-categories equipped with (twisted) actions of $(C_2)^n$ are equivalent.

\subsection*{Acknowledgements}
The authors would like to thank Damien Lejay, Karol Szumi\l{}o, Ben Ward, and the members of the Centre of Australian Category Theory for useful input and suggestions.

\section{\texorpdfstring{Adjoint strings $(L,F,R)$ and right-induced model category structures}{Adjoint strings (L,F,R) and right-induced model category structures}}

A model category $\mathcal{M}$ is called \emph{cofibrantly generated} (see~\cite{Hirschhorn:MCL,Hovey:MC}) if there exist sets of maps $I$ and $J$ so that the class of fibrations of $\mathcal{M}$ is $\rlp{J}$, i.e., the class of maps having the right lifting property with respect to every element of $J$, and the class of acyclic fibrations of $\mathcal{M}$ is $\rlp{I}$.
If $K$ is a set of maps, we write $K\text{-cof} = \llp{(\rlp{K})}$ for those maps having the left lifting property with respect to every map in $\rlp{K}$; in a cofibrantly generated model category, $I\text{-cof}$ is the class of cofibrations and $J\text{-cof}$ is the class of acyclic cofibrations.
We write $K\text{-cell}$ for the closure of $K$ under pushout and transfinite composition.
We will use a quadruple $(\mathcal{M}, \mathcal{W}, I, J)$ to refer to such a cofibrantly generated model category, which allows us to emphasize $\mathcal{W}$, the subcategory of weak equivalences, in addition to the generating sets $I, J$.

Throughout, we will say that a functor $F : \mathcal{N} \to \mathcal{M}$ between model categories \emph{creates weak equivalences} if, for each morphism $f$ of $\mathcal{N}$, we have that $f$ is a weak equivalence in $\mathcal{N}$ if and only if $F(f)$ is a weak equivalence in $\mathcal{M}$.
Similarly, we could discuss when such a functor creates fibrations or cofibrations.

\begin{definition}
Let $F:\mathcal{N}\to\mathcal{M}$ be a functor between model categories. 
If $F$ creates weak equivalences and fibrations then we call the model structure on $\mathcal{N}$ \emph{right-induced} (leaving $F$ implicit). 
We also sometimes say the model structure is \emph{lifted} from $\mathcal{M}$ or along $F$.
\end{definition}

\begin{lemma}\label{lemma: adjoint string, sharp condition}
Suppose that $(\mathcal{M}, \mathcal{W},I,J)$ is a cofibrantly generated model category and $\mathcal{N}$ is a bicomplete category. Let $F:\mathcal{N}\to \mathcal{M}$ be a functor which is right adjoint to $L$ and left adjoint to $R$. Suppose that $FL$ preserves acyclic cofibrations and $LI$ permits the small object argument.

Then $(\mathcal{N},F^{-1}\mathcal{W}, LI,LJ)$ is a right-induced cofibrantly generated model category structure.
\end{lemma}
\begin{proof}
The standard criterion for the existence of a right-induced model structure (e.g.,~\cite[Theorem 11.3.2]{Hirschhorn:MCL}) takes as given a cofibrantly generated model category $(\mathcal{M},\mathcal{W},I,J)$, a bicomplete category $\mathcal{N}$, and an adjoint pair $L:\mathcal{M}\rightleftarrows\mathcal{N}:F$. The theorem says that a right-induced cofibrantly generated model category as in the statement of the lemma exists if maps in $F(LJ\text{-cell})$ are weak equivalences and the sets $LI$ and $LJ$ permit the small object argument.

To see that $F(LJ\text{-cell})$ is in $\mathcal{W}$, since $F$ is a left adjoint, $F(LJ\text{-cell})\subset FLJ\text{-cell}$. Maps in $FLJ$ are acyclic cofibrations by our assumption on $FL$. Since acyclic cofibrations are closed under relative cell complexes, any map in $FLJ\text{-cell}$ is an acyclic cofibration and in particular a weak equivalence.

For the small object argument for $LJ$, let $x$ be the domain of a morphism in $J$. By~\cite[Theorem 10.5.27]{Hirschhorn:MCL}, since $x$ is small relative to $J$, it is $\kappa$-small relative to $J\text{-cof}$ for some $\kappa$. By assumption, $FLJ\subset J\text{-cof}$. Let $z_\beta$ be a $\kappa$-filtered diagram in $LJ\text{-cell}$. Then since $F$ is a left adjoint, $F(z_\beta \to z_{\beta+1})$ is in $FLJ\text{-cell} \subseteq J\text{-cof}$. 
Commutativity of the following diagram shows that $Lx$ is $\kappa$-small with respect to $LJ$.
\[
\begin{tikzcd}
\colim\mathcal{C}(Lx, z_\beta)\ar[rr]\ar[d,"\text{(adjunction) }\wr"']&& \mathcal{C}(Lx,\colim z_\beta)\ar[d,"\wr\text{ (adjunction)}"]\\
\colim{\mathcal{M}}(x,Fz_\beta)\ar[rd,"\sim","\text{(}\kappa\text{-smallness)}"']
&&
\mathcal{M}(x,F\colim z_\beta).\\
&{\mathcal{M}}(x,\colim Fz_\beta)\ar[ur,"\sim","\text{(}F\text{ is a left adjoint)}"']
\end{tikzcd}
\]
\end{proof}
By the usual model category arguments, instead of verifying that $FL$ preserves acyclic cofibrations, one may prove that $FR$ preserves fibrations.

A ``cleaner'' criterion with a more restrictive hypothesis is the following.
\begin{theorem}\label{theorem: adjoint string, Quillen condition}
Suppose that $(\mathcal{M}, \mathcal{W},I,J)$ is a cofibrantly generated model category and $\mathcal{N}$ is a bicomplete category. Let $F:\mathcal{N}\to \mathcal{M}$ be a functor which is right adjoint to $L$ and left adjoint to $R$. 
If $(FL,FR)$ is a Quillen adjunction, then there exists a right-induced cofibrantly generated model category structure $(\mathcal{N},F^{-1}\mathcal{W}, LI,LJ)$. Both $(L,F)$ and $(F,R)$ are Quillen adjunctions.
\end{theorem}
\begin{proof}
Since the functor $FL$ is left Quillen, it preserves acyclic cofibrations. The argument used in the proof of Lemma~\ref{lemma: adjoint string, sharp condition} for $LJ$ shows that $LI$ permits the small object argument.

To see that $F$ is left Quillen, let $f$ be a fibration. Since $FR$ is right Quillen, $FR(f)$ is a fibration in $\mathcal{M}$. Since $F$ reflects fibrations, $Rf$ is a fibration in $\mathcal{N}$. 
The same argument applies when $f$ is an acyclic fibration, hence $R$ is right Quillen.
\end{proof}

Suppose we are in the setting of Lemma~\ref{lemma: adjoint string, sharp condition} where $(F,R)$ is not necessarily a Quillen adjunction but we know that $FR$ preserves fibrations. Then the argument in the second paragraph in this proof still applies to show that $R$ preserves fibrations and $F$ preserves acyclic cofibrations.

It is standard that any right-induced model structure over a right proper base is also right proper. The following proposition indicates that we are working in a very special setting.

\begin{proposition}\label{prop: left proper}
In the context of Theorem~\ref{theorem: adjoint string, Quillen condition}, if $\mathcal{M}$ is left proper, so is $\mathcal{N}$.
\end{proposition}
\begin{proof}
The functor $F$ reflects weak equivalences and preserves pushouts, weak equivalences, and cofibrations.	
\end{proof}

\begin{remark}
A natural question is whether the criteria of Lemma~\ref{lemma: adjoint string, sharp condition} and Theorem~\ref{theorem: adjoint string, Quillen condition} have dual versions for \emph{left}-induced model category structures 
(studied, e.g., in~\cite{BayehHessKarpovaKedziorekRiehlShipley:LIMSDC,HessKedziorekRiehlShipley:NSCIMS}). 
That is, given a functor $F:\mathcal{N}\to \mathcal{M}$ with both a left adjoint $L$ and a right adjoint $R$, can one make a model category where the weak equivalences and \emph{cofibrations} are created by $F$?
Our methods do not seem to apply in this case.
Typically for the existence of an induced model structure one needs an acyclicity condition and factorizations. 
In the right-induced case the cofibrations whose acyclicity must be tested are $LJ$-cellular, where $J$ is a class of maps in $\mathcal{M}$.
The fundamental property of $F$ that we use is that it takes relative $LJ$-cell complexes to relative $FLJ$-cell complexes.
In the left-induced case, one has two choices. 
One could try to work with a dual condition to cellularity (a Postnikov tower), in which case there is not a general way to build the desired factorizations. 
Alternatively, one can use arguments that show the existence of factorizations~\cite{MakkaiRosicky:CC,BourkeGarner:AWFSI}; 
in this case the fibrations whose acyclicity must be tested are not necessarily $R\mathcal{Z}$-Postnikov for any class $\mathcal{Z}$ of maps in $\mathcal{M}$.
\end{remark}

Here we give a few examples of the application of Theorem~\ref{theorem: adjoint string, Quillen condition} recovering some known model structures and demonstrating some limitations of this tool.

\subsection{Groupoids and categories} 
The category $\cat$ of small categories and functors between them admits a model structure, called the canonical model structure, where the weak equivalences are the categorical equivalences, the cofibrations are the injective-on-objects functors, and the fibrations are the isofibrations~\cite{JoyalTierney:SSCS}.
This model structure is cofibrantly generated~\cite{Rezk:MCC},\cite[Example 1.1]{Lack:QMS2C}.
The inclusion $F$ of groupoids into $\cat$ has both a left adjoint $L$ and a right adjoint $R$ (this is standard; see, e.g.,~\cite[Example 4.1.15]{Riehl:CTC} for an explicit description of the adjoints). It is a direct observation that $FL$ preserves cofibrations and that $FR$ preserves fibrations, so the canonical model structure on groupoids is right-induced. This example is ahistorical: the existence of the canonical model structure on groupoids predates that for categories in the literature~\cite{Anderson:FGR} (moreover, showing the existence of the structure uses very similar arguments for $\cat$ and for $\gpd$). 

\subsection{Categories with anti-involution}\label{section categories with anti involution}
Let $\catinv$ be the category whose objects are small categories $X$ together with a functor $\tau: X^{\op} \to X$ so that $\tau^{\op} \tau = \id_{X^{\op}}$. Morphisms are functors satisfying $\tau f^{\op} = f \tau$. 

The forgetful functor $F: \catinv \to \cat$ has both a left adjoint $L$ and a right adjoint $R$. 
The left adjoint takes $X$ to the category $X\amalg X^{\op}$ equipped with the swap map
\[ [X \amalg X^{\op}]^{\op} = X^{\op} \amalg X \to X \amalg X^{\op} = FL(X)\]
and the right adjoint takes $X$ to the category $X\times X^{\op}$ with the swap map
\[ [X \times X^{\op}]^{\op} = X^{\op} \times X \to X \times X^{\op} = FR(X).\]
To see that $L$ is left adjoint to $F$, notice that
\begin{align*} \hom(L(X),(Y, \tau))  &= \{ f \amalg g : X \amalg X^{\op} \to Y \,|\, \tau f^{\op} = g \} \\ &= \{ f \amalg \tau f^{\op} : X \amalg X^{\op} \to Y \} \cong \hom(X, Y). \end{align*}
A similar argument shows that $R$ is right adjoint to $F$.

To see that the adjoint string $L \dashv F \dashv R$ fits into our framework and endows $\catinv$ with a right-induced model structure, we simply note that $FL(X) = X \amalg X^{\op}$ preserves the cofibrations and every equivalence of categories. Thus $FL$ is left Quillen.

\begin{remark}
The functor $FL$ is a 2-monad on $\cat$, and the 2-category of $FL$-algebras admits a $\cat$-enriched model structure by~\cite[Theorem 5.5]{Lack:HTA2M}. The underlying 1-category of this 2-category is just $\catinv$, and the underlying ordinary model structure is the one guaranteed by Theorem~\ref{theorem: adjoint string, Quillen condition}.
\end{remark}

\begin{exercise}
	The cofibrations in $\catinv$ are those maps $f: (X, \tau') \to (Y,\tau)$ so that $f$ is injective on objects and $\tau$ acts freely on $(\ob\,Y)\setminus f(\ob\,X)$.
\end{exercise}

\begin{remark}
The category $\catinv$ describes categories with strict anti-involution. 
One might wonder about categories with strong ``anti-involution,'' where the composition $\tau^{\op}\tau$ is only naturally isomorphic to the identity. 
There is a functorial strictification of strong anti-involutive categories (see, for instance,~\cite[\S 3]{Dotto:ECFBZ2ARAK}) so it is reasonable to study the strict version in lieu of the strong version.
\end{remark}

\subsection{Dagger categories}
Dagger categories are categories equipped with anti-involution $\tau$ which is the identity on objects. Bunke proved~\cite{Bunke:HTSC} that there is a model structure on $\dcat$ with weak equivalences the \emph{dagger equivalences}; this agrees with the underlying ordinary model structure of the $\cat$-model structure on $\dcat$ guaranteed by~\cite[Theorem 4.3]{Lack:HTA2M}.

It turns out that $\dcat$ is both reflective and coreflective in $\catinv$.
Let us give an indication of the adjoints $L,R: \catinv \to \dcat$.
The functor $\tau$ induces an automorphism $\ob\,\tau$ of $\ob\,X$. 
The object set of the underlying category of $L(X,\tau)$ is the set of orbits $(\ob\,X)/(\ob\,\tau)$ (we will not need to describe the morphisms or involution explicitly).
The underlying category of $R(X,\tau)$ is the full subcategory with objects the fixed-points of $\ob\,\tau : \ob\,X \to \ob\,X$. 
The anti-involution on morphisms is then the restriction of $\tau$ to this full subcategory.
The counit map $R(X,\tau) \to (X,\tau)$ is the full subcategory inclusion. 
That these maps are adjoints follows, essentially, from Theorem 3.1.17 and Theorem 3.1.19 of~\cite{Heunen:CQML}.

The functor $FR$ does not preserve fibrations. Consider the categories $X$ and $Y$ with $\ob\,X = \{x, x', y\}$, $\ob\,Y = \{z,y\}$, and $\hom(a,b) = *$ for any pair of objects $a,b$ in $X$ or in $Y$.
Consider these categories as objects in $\catinv$ so that the anti-involutions fix $y$, swap $x$ and $x'$, and fix $z$.
The functor $p: X \to Y$ with $p(x) = p(x') = z$ and $p(y) = y$ is an isofibration, but $RX = \{y\} \to RY = Y$ is not an isofibration.

Since $FR$ does not preserve fibrations, $FL$ does not preserve acyclic cofibrations and Lemma~\ref{lemma: adjoint string, sharp condition} does not apply.

We mention in passing that if we endow $\dcat$ with the Bunke model structure, then $\dcat \hookrightarrow \catinv$ is neither right nor left Quillen.
Since only the initial dagger category is cofibrant as an object in $\catinv$, the inclusion cannot be left Quillen.
The inclusion is not right Quillen because the forgetful functor $\dcat \to \cat$ is not: 
if $X$ is any dagger category, the inclusion of the subcategory of unitary isomorphisms is a fibration in $\dcat$; it is only a fibration in $\cat$ if every isomorphism in $X$ is unitary.

\subsection{Simplicial categories with anti-involution}\label{section: simpcat}
Let $\simpcat$ denote the category of simplicially-enriched (small) categories, equipped with the Bergner model structure~\cite{Bergner:MCSCSC}, which blends together the canonical model structure on $\cat$ with the Kan--Quillen model structure on $\sset$, the category of simplicial sets. In this model structure, a functor $f: X \to Y$ is a weak equivalence (resp. fibration) if
\begin{itemize}
	\item the functor $\pi_0(f) : \pi_0(X) \to \pi_0(Y)$ is an equivalence (resp. isofibration) of ordinary categories 
	\item for each pair of objects $c,c'$, the map $X(c,c') \to Y(fc,fc')$ of simplicial sets is a weak equivalence (resp. fibration) in the Kan--Quillen model structure.
\end{itemize}
These weak equivalences are called \emph{Dwyer--Kan equivalences}~\cite{DwyerKan:FCHA}.

An \emph{anti-involutive simplicial category} is a simplicial category $X$ together with a functor
$\tau: X^{\op} \to X$ (where $X^{\op}$ denotes the simplicial category with the same objects and $X^{\op}(A,B)=X(B,A)$) so that $\tau^{\op} \tau = \id_{X^{\op}}$. 
We will write $\simpcatinv$ for the collection of such objects together with morphisms those simplicial functors satisfying $\tau f^{\op} = f \tau$.

The forgetful functor $F: \simpcatinv \to \simpcat$ has both a left adjoint $L$ and a right adjoint $R$, given by the same formulas as in Section~\ref{section categories with anti involution}. 
Cofibrations in $\simpcat$ are not as simple as they are in $\cat$. 
However, it is straightforward to show that $FR$ preserves (acyclic) fibrations of simplicial categories and thus is right Quillen. 
Thus, by Theorem~\ref{theorem: adjoint string, Quillen condition}, $\simpcatinv$ admits a model structure lifted along $F$.

\subsection{Chain complexes}
Let $f:R\to S$ be a morphism of rings. Then there is a pullback functor $f^*$ from $S$-modules to $R$-modules which has both a left adjoint $f^*(S)\otimes_R(-)$ and a right adjoint $\mor_{R\text{-mod}}(f^*(S),-)$~\cite[A.5.2.2(e)]{Eisenbud:CAWVTAG}.
These functors and adjunctions upgrade to adjunctions between
chain complexes of $S$-modules (henceforth ``$S$-chain complexes'') and $R$-chain complexes, by applying each functor degreewise. 
The category of $R$-chain complexes supports projective and injective model structures, both of which are cofibrantly generated~\cite[Theorems 2.3.11, 2.3.13]{Hovey:MC}. In both cases the weak equivalences are quasi-isomorphisms. In the projective model structure, the fibrations are degreewise epimorphisms and in the injective model structure the cofibrations are degreewise monomorphisms.

Suppose that $f^*S$ is $R$-projective. Then $\mor_{R\text{-mod}}(f^*S,-)$ is an exact functor and so $f^*(\mor_{R\text{-mod}}(f^*S,-))$ preserves fibrations and weak equivalences in the projective model structure and thus is right Quillen. We conclude that the projective model structure lifts along $f^*$ to $S$-chain complexes. Because epimorphisms in $S$-chain complexes and $R$-chain complexes are both calculated in sets, and weak equivalences in both are calculated in $\mathbb{Z}$-chain complexes, the right-induced model structure is the projective model structure on $S$-chain complexes.

This example is a bit silly because we already knew that $S$-chain complexes support a projective model structure and it is well-known (c.f.~\cite[end of section~2.3]{Hovey:MC}) that $f^*$ is right Quillen. Indeed, even without the condition that $f^*S$ be a projective $R$-module, the functor $f^*$ is right Quillen and creates both fibrations and weak equivalences. This is evidence that Theorem~\ref{theorem: adjoint string, Quillen condition} is fairly weak.

A less familiar example comes from the injective model structure.
Suppose that $f^*S$ is $R$-flat. Then the functor $f^*(S) \otimes_R (-) $ is exact which implies that the functor $f^*(f^*(S)\otimes_R (-))$ preserves cofibrations and weak equivalences in the injective model structure on $R$-chain complexes and so is left Quillen. We conclude that the injective model structure lifts along $f^*$ to $S$-chain complexes.

However, the resulting model structure on $S$-chain complexes is not in general the injective model structure. The fibrations in the injective model structure are the epimorphisms with fibrant kernel. The fibrant objects in the injective structure on $S$-chain complexes are a subclass of complexes of $S$-injective chain complexes containing the bounded complexes. However, a map in the right-induced structure is a fibration if and only if it is an epimorphism whose kernel is fibrant in the $R$-injective model structure.

\subsection{Non-negatively graded cochain complexes}
An example where the hypotheses of Theorem~\ref{theorem: adjoint string, Quillen condition} are not satisfied but Lemma~\ref{lemma: adjoint string, sharp condition} is useful is the following. Let $F$ be the inclusion of non-negatively graded cochain complexes into unbounded cochain complexes. 
The functor $F$ has adjoints on both sides.
The right adjoint $R$ naively truncates by discarding components in negative degree and keeping components in non-negative degree.
The left adjoint $L$ truncates ``homotopically'' --- it discards components in negative degree, keeps components in positive degree, but in degree zero $L$ takes the cokernel of the differential from degree $-1$ to degree $0$.
We would like to lift the projective model structure from unbounded cochain complexes to non-negatively graded cochain complexes along $F$ (that this is possible is well-known). 

However, in order to use Theorem~\ref{theorem: adjoint string, Quillen condition} to prove this fact would require $FR$ to be right Quillen. This is false for any nonzero ring. For example, consider the unbounded cochain complex $C$ which is the ground ring in degrees $0$ and $-1$ with the identity map as differential. The map from $C$ to $0$ is an acyclic fibration but applying $FR$ to it gives the map from the ground ring to zero.

On the other hand, in order to use Lemma~\ref{lemma: adjoint string, sharp condition} we need only verify a smallness condition and that $FR$ preserves fibrations (degreewise epimorphisms). Every object in cochain complexes is small. The functor $FR$ merely discards all modules in negative degree and so preserves degreewise epimorphisms.
\subsection*{Operads and their variations}
We conclude this section with two examples about operad-like structures. Our examples relate operads~\cite{May:GILS}, cyclic operads~\cite{GetzlerKapranov:COCH}, and modular operads~\cite{GetzlerKapranov:MO}. We will follow the conventions of the book~\cite{MarklShniderStasheff:OATP}.  Readers unfamiliar with operads may wish to skip to Section~\ref{sec: functor categories}.

\subsection{Cyclic operads and operads}
Operads are themselves algebras for an $\mathbb{N}$-colored operad $\mathbb{O}$ of rooted planar trees (see~\cite[1.5.6]{BergerMoerdijk:RCORHA}).
There is an $\mathbb{N}$-colored operad $\mathbb{C}$ (in $\set$) whose algebras are cyclic operads (see~\cite[2.3.1]{KaufmannWard:FC} coupled with~\cite{BataninKockWeber:RPSFCO}; or a variation of~\cite[14.1.1]{YauJohnson:FPAM}). 
The elements of $\mathbb{C}$ are unrooted planar trees.
Forgetting the root gives a map $\mathbb{O} \to \mathbb{C}$, which induces the forgetful functor from cyclic operads to operads.
As is always the case with maps of colored operads, there is a left adjoint to this forgetful functor.
The forgetful functor also has a right adjoint.
Existence of this right adjoint is apparently due to Templeton in his unpublished thesis~\cite{Templeton:SGO} and was rediscovered independently by Ward~\cite{Ward:6OFGO}. Ward gives an explicit construction in the category of chain complexes over a field and an argument that works for some other ground categories. We record the result in the general case.

\begin{lemma}\label{lemma cyclic right adjoint}
Suppose $\mathcal{E}$ is a symmetric monoidal category with finite limits. Then the forgetful functor from cyclic operads in $\mathcal{E}$ to operads in $\mathcal{E}$ admits a right adjoint $R$ with the underlying $\mathbb{N}$-module
\[ RP(n) = \prod_{i=0}^n P(n).\]	
\end{lemma}

\begin{proof}
We describe the (cyclic) operad structure on $RP$.
For $1\le i\le m$, define $\circ_i : RP(m) \otimes RP(n) \to RP(m+n-1)$ by
\[
	\pi_j (\circ_i) = \begin{cases}
\circ_{i+j} (\pi_{j} \otimes \pi_0) &0\le j \le m-i
\\
\circ_{i+j-m} (\pi_{i+j-m} \otimes \pi_{m+1-i}) \tau &m-i < j\le m+n-i
\\
\circ_{i+j-m-n} (\pi_{j-n+1} \otimes \pi_0) & m+n-i< j<m+n.
\end{cases}
\]
where $\tau$ is the symmetry constraint of $\mathcal{E}$.

For the extended symmetric group action, if $\sigma \in \Sigma_{n+1} = \aut(\{0,1,\dots, n\})$ (thought of as acting on integers modulo $(n+1)$) and $0 \leq i \leq n$, let $\sigma_i \in \Sigma_{n} = \aut(\{1,\dots, n\})$ be given by $k\mapsto \sigma(k-i) - \sigma(n+1-i) \mod (n+1)$.
Using the $\Sigma_n$ action on $P(n)$, define $\sigma^* : RP(n) \to RP(n)$ by $\pi_i \sigma^* \coloneqq \sigma_i^*\pi_{n+1 - \sigma(n+1-i)}.$

The identity $\id_{RP}$ is $\id_{P} \times \id_P : \mathbf{1} \to P(1) \times P(1).$

Now the proof that $RP$ is a cyclic operad and that $R$ is right adjoint to the forgetful functor is tedious but straightforward. 
\end{proof}
Neglect of structure gives similar results for non-symmetric operads and non-unital Markl operads.

In light of~\cite[Theorem 2.1]{BergerMoerdijk:RCORHA}, the category of cyclic operads in a suitably nice symmetric monoidal model category $\mathcal{M}$ admits a model structure lifted along the forgetful functor to $\mathbb{N}$-modules. 
We can recover this result from the existence of a model structure on operads.

\begin{proposition}
\label{prop: cyclic operads}
Let $\mathcal{M}$ be a symmetric monoidal model category such that operads in $\mathcal{M}$ have a cofibrantly generated model category structure lifted from $\mathbb{N}$-modules in $\mathcal{M}$. Then cyclic operads in $\mathcal{M}$ have a model category structure lifted from operads (or $\mathbb{N}$-modules) in $\mathcal{M}$.
\end{proposition}
\begin{proof}
We use Theorem~\ref{theorem: adjoint string, Quillen condition} to lift the model structure on operads to cyclic operads. 
Cyclic operads, as algebras over a colored operad, are bicomplete. Because the model structure on operads is right-induced, fibrations and acyclic fibrations are calculated in $\mathbb{N}$-modules.
If $f$ is any map of cyclic operads, then the degree $n$ part is $FR(f)(n) = \prod_{n+1} f(n)$. As products preserve both fibrations and acyclic fibrations, if $f$ is an (acyclic) fibration then so is $FR(f)$.
\end{proof}

\subsection{Modular and cyclic operads} 
The functor $F$ from modular operads to cyclic operads (with some generality as to the ground category $\mathcal{M}$), which takes a modular operad to its genus zero part has a left adjoint $L$ (modular envelope) and a right adjoint $R$ (extension by the terminal object of $\mathcal{M}$)~\cite{Ward:6OFGO}. 

Suppose that cyclic operads in $\mathcal{M}$ support a cofibrantly generated model category structure. The compositions $FL$ and $FR$ are both the identity and therefore Theorem~\ref{theorem: adjoint string, Quillen condition} applies. Then there is a right-induced model structure on modular operads where weak equivalences are created in cyclic operads after forgetting higher genus operations.

One would usually prefer a model structure on modular operads where weak equivalences are created in $\mathcal{M}^{\mathbb{N}}$ by considering the coproduct over genera in each arity. 
The forgetful functor creating such weak equivalences factors through a \emph{different} forgetful functor to cyclic operads that forgets the information of the genus but not the higher genus operations themselves. However, this less-forgetful functor does not have a right adjoint and so our method does not apply in this case. This does not preclude the proof of existence of such a model structure by other methods. See, e.g.,~\cite[Proposition~10.3]{BataninBerger:HTAOPM}~and~\cite[Theorems~8.15~and~8.27]{KaufmannWard:FC}.

\section{Functor categories}
\label{sec: functor categories}

Let $\ground$ be a bicomplete category and $\mathcal{C}$ and $\mathcal{D}$ be small categories. Any functor $\iota:\mathcal{C}\to \mathcal{D}$ induces a functor $\iota^*$ between the functor categories $\ground^{\mathcal{D}}\to\ground^{\mathcal{C}}$. Bicompleteness of $\ground$ implies that $\iota^*$ has both a left adjoint $\iota_!$ (given by left Kan extension) and a right adjoint $\iota_*$ (given by right Kan extension). For the criterion of Lemma~\ref{lemma: adjoint string, sharp condition}, what is important to have is an explicit description of the composite endofunctors $\iota^*\iota_!$ and $\iota^*\iota_*$. 
If $\ground$ is a bicomplete category and $\iota$ is a functor between small categories $\mathcal{C}$ and $\mathcal{D}$, then we have 
(see, for instance,~\cite[Theorem 1, X.3]{MacLane:CWM}) the following canonical isomorphisms for $d\in \mathcal{D}$:
\begin{align*}
(\iota_!X)(d)&\cong \colim_{\iota\downarrow d}X
\\
(\iota_*X)(d)&\cong \lim_{d\downarrow \iota}X. 
\end{align*}
In practice, analysis of the indexing categories $\iota\downarrow \iota(c)$ and $\iota(c)\downarrow \iota$ helps verify that $\iota^*\iota_!$ and $\iota^*\iota_*$ satisfy the criteria of Lemma~\ref{lemma: adjoint string, sharp condition}.

In the context of functor categories, we can use the following smallness criterion to apply Lemma~\ref{lemma: adjoint string, sharp condition}.
\begin{lemma}\label{lemma: all objects small preserved by exp}
Let $\mathcal{C}$ be a small category and let $\groundcat$ be a cocomplete category such that every object is small. Then every object of $\groundcat^{\mathcal{C}}$ is small.
\end{lemma}
For convenience, we include a proof in Appendix~\ref{appendix: proof of lemma}, but this lemma is more or less well-known when $\groundcat$ is locally presentable. In a locally presentable category, every object is small~\cite[Proposition 1.16]{AdamekRosicky:LPAC}; further, if $\groundcat$ is locally presentable and $\mathcal{C}$ is a small category, then $\groundcat^\mathcal{C}$ is locally presentable as well~\cite[Corollary 1.54]{AdamekRosicky:LPAC}.
The hypothesis on $\groundcat$ in the lemma is strictly weaker than local presentability, as can be seen in~\cite[Remark 1.4]{Beke:SHMC}.

An easy case in functor categories is the following.
\begin{corollary}
Let $\iota:\mathcal{C}\to\mathcal{D}$ be a fully faithful functor with $\mathcal{C}$ non-empty. Assume there is a cofibrantly generated model category structure on the functor category $\fun(\mathcal{C},\ground)$. Then there is a model category structure lifted along $\iota^*$ on the functor category $\fun(\mathcal{D},\ground)$.
\end{corollary}
\begin{proof}
Because $\mathcal{C}$ is non-empty, bicompleteness of the functor category $\fun(\mathcal{C},\ground)$ implies bicompleteness of $\ground$, which then implies bicompleteness of the functor category $\fun(\mathcal{D},\ground)$.

Next, $\iota \downarrow \iota(c)$ has a terminal object and $\iota(c)\downarrow \iota$ has an initial object, in both cases the identity of $c$. Then both $\iota^*\iota_!$ and $\iota^*\iota_*$ are the identity functor, and thus form a Quillen adjunction for any model category structure on the functor category $\fun(\mathcal{C},\ground)$. Then Theorem~\ref{theorem: adjoint string, Quillen condition} applies.
\end{proof}
\begin{remark}
A model structure on a functor category created in this way should not be expected to carry very interesting information about $\mathcal{D}$ because it entirely ignores any structure whatsoever on objects outside of the image of $\iota$.
\end{remark}
Examples include the inclusion of the simplicial category into the augmented simplicial category or to (uniformly bounded) ordinal and poset categories. Another example is the inclusion of the dendroidal category $\Omega$ of~\cite{MoerdijkWeiss:DS} into the graphical category $\Gamma$ of~\cite{HackneyRobertsonYau:IPIWP}. Again, all of the resulting model category structures thus created are pathological and in each case the weak equivalences are ``wrong.''

Turning away from fully faithful functors, this setup also recovers other well-known and important model structures. We omit the detailed computation that the condition of Lemma~\ref{lemma: adjoint string, sharp condition} is satisfied; in the following cases this is more or less a repackaging of the standard argument (see e.g.~\cite[11.5--11.6]{Hirschhorn:MCL}).
\begin{example}\label{example: group actions}
Let $\iota:G\to H$ be an inclusion of groups and $\mathcal{M}$ a cofibrantly generated model category. Then $\iota$ induces an adjoint string $\iota_!\dashv\iota^*\dashv \iota_*$ between $\mathcal{M}^G$ and $\mathcal{M}^H$ satisfying our criterion and the naive model structure on $G$-objects in $\mathcal{M}$ lifts to the naive model structure on $H$-objects in $\mathcal{M}$.
\end{example}
\begin{example}
Let $\mathcal{C}$ be a small category and $\mathcal{M}$ a cofibrantly generated model category. The inclusion $\iota$ of the objects of $\mathcal{C}$, considered as a discrete category, into $\mathcal{C}$ induces an adjoint string $\iota_*\dashv\iota^*\dashv \iota_*$ between $\prod_{\ob\,\mathcal{C}}\mathcal{M}$ and $\mathcal{M}^C$ satisfying our criterion, which then yields the projective model structure on diagrams in $\mathcal{M}^{\mathcal{C}}$.
\end{example}
\begin{example}
\label{example: crossed simplicial groups}
Let $\iota:\Delta\to \Delta G$ be the inclusion of the simplicial category into a crossed simplicial group~\cite{FiedorowiczLoday:CSGAH}. 
Then $\iota$ induces an adjoint string $\iota_!\dashv\iota^*\dashv \iota_*$ between $\set^{\Delta^{\op}}$ and $\set^{(\Delta G)^{\op}}$. 
A direct calculation of $\iota^*\iota_!$ shows that it preserves monomorphisms; the characterization of~\cite[Proposition~5.1]{FiedorowiczLoday:CSGAH} shows that it preserves Kan--Quillen weak equivalences. 
Therefore there is a right-induced Kan--Quillen type model structure on $(\Delta G)$-sets. 
This example goes back to~\cite[Theorem~6.2]{DwyerHopkinsKan:HTCS}; see also~\cite{Balchin:HPLGEP}.
\end{example}
\section{Semidirect products}
\begin{definition}
Let $\mathcal{C}$ be a category, $\aut(\mathcal{C})$ be the group of isomorphisms $\mathcal{C} \overset\cong\to \mathcal{C}$, and $G\xrightarrow{\rho} \aut(\mathcal{C})$ be an action of $G$ on the category $\mathcal{C}$. Then the \emph{semidirect product} $\mathcal{C} \rtimes G$ has objects the objects of $\mathcal{C}$ and morphisms pairs $(\varphi,g)$ with $\varphi$ a $\mathcal{C}$-morphism and $g$ a group element; if the source and target of $\varphi$ are $x$ and $y$ then the source and target of $(\varphi, g)$ are $\rho_{g^{-1}}(x)$ and $y$. Composition is given by
\[
(\varphi, g)\circ (\psi,h)=(\varphi\circ \rho_g(\psi), gh).
\]
\end{definition}
Given a group $G$ which acts on the category $\mathcal{C}$, there is a functor $\iota:\mathcal{C}\to \mathcal{C}\rtimes G$ which is the identity on objects and $\varphi\mapsto (\varphi,1)$ on morphisms.
\begin{lemma}\label{lemma: Lan for semidirect products}
Let $G$ be a group acting on a small category $\mathcal{C}$. Let $F:\mathcal{C}\to \groundcat$ be a functor with $\groundcat$ cocomplete. Then there is a natural isomorphism 
\[
\iota^*\iota_!(F) \cong \coprod_{g\in G} (\rho_{g^{-1}})^*F
\]
of functors $\mathcal{C} \to \groundcat$.
\end{lemma}
\begin{proof}
Since $\iota_!$ is the left Kan extension, the functor $\iota^*\iota_!$ is calculated for $F : \mathcal{C} \to \groundcat$ and $x$ an object of $\mathcal{C}$ as
\[\iota_!(F)(x) \cong \colim_{\iota\downarrow{}x} F(y).\]
Any object $y\to x$ in $\iota\downarrow{} x$ factors uniquely as a map in $\mathcal{C}$ followed by a morphism of the form $(\id_x, g)$ for some $g\in G$.
This partitions the comma category into a disjoint union indexed by $g\in G$; moreover the object $(\id_x, g): \rho_{g^{-1}}x \to x$ is terminal in its connected component. Then
\[
\iota^*\iota_!(F)(x) \cong \coprod_{g\in G} F(\rho_{g^{-1}}(x)) = \coprod_{g\in G} \left[(\rho_{g^{-1}})^*F\right] (x).
\]
This shows the lemma at the level of objects. Then $\iota^*\iota_!(F)$ applied to a map $\psi$ is calculated as the map between these respective coproducts induced by $\psi$. A quick calculation using the commutation relation in the semidirect product then implies that the induced map is $\psi$ twisted by $\rho_{g^{-1}}$ on the factor of the coproduct indexed by $g$.
\end{proof}

\begin{theorem}\label{theorem: semidirect}
Suppose $G$ is a group which acts on a small category $\mathcal{C}$, let $\groundcat$ be a bicomplete category, and suppose that $\groundcat^{\mathcal{C}}$ has a cofibrantly generated model category structure. Suppose one of the following conditions hold:
\begin{enumerate}
\item all objects of $\groundcat$ are small and for all $g$ in $G$, $(\rho_g)^*$ preserves acyclic cofibrations, or
\item for all $g$ in $G$, $(\rho_g)^*$ is a left Quillen functor.
\end{enumerate}
Then $\groundcat^{\mathcal{C}\rtimes G}$ has a cofibrantly generated model category structure lifted from $\groundcat^{\mathcal{C}}$.
\end{theorem} 
\begin{remark}
Note that since $\rho_g$ is an isomorphism of categories with inverse $\rho_{g^{-1}}$, the condition that $(\rho_g)^*$ preserves acyclic cofibrations for all $g$ is equivalent to the condition that $(\rho_g)^*$ preserves fibrations for all $g$. Similarly, the condition that $(\rho_g)^*$ is left Quillen for all $g$ implies that $(\rho_g)^*$ is also right Quillen. 
\end{remark}
\begin{proof}
The functor $\iota$ yields an adjoint string $\iota_!\dashv \iota^*\dashv \iota_*$, with $\iota^*: \groundcat^{\mathcal{C}\rtimes G}\to \groundcat^{\mathcal{C}}$ and $\iota_*$ and $\iota_!$ in the other direction $\groundcat^{\mathcal{C}}\to \groundcat^{\mathcal{C}\rtimes G}$, calculated as the right and left Kan extensions along $\iota$. Limits and colimits exist in the functor category by bicompleteness of $\groundcat$ (and are calculated pointwise).

In the first case, the functor $(\rho_{g^{-1}})^*$ preserves acyclic cofibrations, as does the coproduct over $G$. Then $\iota^*\iota_!$, computed in Lemma~\ref{lemma: Lan for semidirect products}, preserves acyclic cofibrations. By Lemma~\ref{lemma: all objects small preserved by exp}, all objects of the functor category are small. Thus the hypotheses of Lemma~\ref{lemma: adjoint string, sharp condition} are satisfied.

In the second case, since $(\rho_{g^{-1}})^*$ is a left Quillen adjoint, the argument used in the first case for acyclic cofibrations also applies to cofibrations. Then $\iota^*\iota_!$ is left Quillen and we are done by Theorem~\ref{theorem: adjoint string, Quillen condition}.
\end{proof}

\begin{remark}
In the above discussion, one could instead ask about \emph{presheaves} on $\mathcal{C} \rtimes G$.
This fits into the same framework as follows.
Suppose that $\rho: G\to \aut(\mathcal{C})$ is a group homomorphism, and define an auxillary homomorphism $\kappa: G^{\op} \to \aut(\mathcal{C}^{\op})$ by $\kappa_g = \rho_{g^{-1}}^{\op} : \mathcal{C}^{\op} \to \mathcal{C}^{\op}.$
We can then form two semidirect products: $\mathcal{C} \rtimes_\rho G$ and $\mathcal{C}^{\op} \rtimes_\kappa G^{\op}.$
There is a contravariant functor 
\begin{align*}
	\mathcal{C} \rtimes_\rho G &\to \mathcal{C}^{\op} \rtimes_\kappa G^{\op} \\
	(\varphi, g) &\mapsto (\rho_{g^{-1}}^{\op}(\varphi^{\op}), g^{\op})
\end{align*}
which exhibits an isomorphism $(\mathcal{C} \rtimes_\rho G)^{\op} \cong \mathcal{C}^{\op} \rtimes_\kappa G^{\op}$. We have the following corollary.
\end{remark}

\begin{corollary}\label{cor: semidirect}
Suppose $G$ is a group which acts on a small category $\mathcal{C}$, let $\groundcat$ be a bicomplete category, and suppose that $\groundcat^{\mathcal{C}^{\op}}$ has a cofibrantly generated model category structure. Suppose one of the following conditions hold:
\begin{enumerate}
\item all objects of $\groundcat$ are small and for all $g$ in $G$, $(\rho_g)^*$ preserves acyclic cofibrations, or
\item for all $g$ in $G$, $(\rho_g)^*$ is a left Quillen functor.
\end{enumerate}
Then $\groundcat^{(\mathcal{C}\rtimes G)^{\op}}$ has a cofibrantly generated model category structure lifted from $\groundcat^{\mathcal{C}^{\op}}$.
\end{corollary}

Now we apply Theorem~\ref{theorem: semidirect} to give examples of right-induced model structures.
\begin{example}
Let $\set^{\mathcal{C}}$ have a cofibrantly generated model category structure and let $G$ be a group. The identity functor is always left Quillen so $\set^{\mathcal{C}\times G}$ (here $\mathcal{C}\times G$ is the semidirect product with the trivial action) has a right-induced model category structure. This is the model category structure on $G$-objects in $\mathcal{M} = \set^{\mathcal{C}}$, previously mentioned in Example~\ref{example: group actions}.
\end{example}

\begin{definition}
The category $\nabla$ is the semidirect product $\Delta\rtimes C_2$, where $C_2$ acts on $\Delta$ by the involution $\mathcal{F}$ sending a function $f:[m]\to [n]$ to $\tau^m \circ f\circ \tau^n$, where $\tau^n$ is the order-reversing bijection on $[n]$.
\end{definition}

Here is an alternative definition of the category $\nabla$.
If $f: [m] \to [n] = \{0, 1, \dots, n\}$ is a map of sets, say $f$ is \emph{monotone} if it is either weakly increasing or weakly decreasing.
The collection of monotone maps forms a subcategory of $\set$, but it is not quite the category $\nabla$: the category $\nabla$ has two distinct automorphisms of $[0]$. 
If $f$ is a monotone map and there exists an $i$ with $f(i) < f(i+t)$, then define $\sgn(f) = t \in \{+1, -1\}$.
Then the category $\nabla$ is isomorphic to the category with morphisms pairs $(f,t)$ with $f:[m] \to [n]$ monotone, $t\in \{+1, -1\}$, so that $t = \sgn(f)$ when $f$ is not constant, and $(f,t) \circ (f',t') = (f\circ f', tt')$.
   
\begin{corollary}\label{corollary: joyal}
There is a Joyal model category structure on $\set^{\nabla^{\op}}$ lifted along the restriction functor from the Joyal model category structures on simplicial sets.
\end{corollary}
\begin{remark}
Example~\ref{example: crossed simplicial groups} gives a Kan--Quillen model structure in the more general setting of crossed simplicial groups. 
For all of the other basic crossed simplicial groups appearing in the classification of~\cite{FiedorowiczLoday:CSGAH}, there is no distinction between inner and outer horns.
It thus seems unlikely that there is a right-induced Joyal-type structure in the same level of generality.
\end{remark}
\begin{proof}
The group action of $C_2$ on simplicial sets (flipping the order of simplices) preserves the set of boundaries so it preserves boundary inclusions of simplices.
Then it preserves all cofibrations. 
A map $f$ is a Joyal equivalence if and only if $\hcr(f)$ is a weak equivalence in $\simpcat$ (see, for instance,~\cite[Theorem 2.2.5.1]{Lurie:HTT}). 
There is a natural isomorphism $\hcr(A^{\op})\cong \hcr(A)^{\op}$ and Dwyer--Kan equivalences are obviously closed under taking opposites. 
Therefore the group action also preserves weak equivalences.
\end{proof}

Notice that the cofibrations are generated by $\iota_!(\partial \Delta^n \to \Delta^n)$, that is, they are \emph{normal monomorphisms}.
Explicitly, a map $X \to Y$ is a normal monomorphism if it is a levelwise monomorphism and if the $C_2$ action on $Y_n \setminus X_n$ is free. Alternatively, one may check this only on non-degenerate simplices.

Along the same lines, we have the following.

\begin{example}\label{example: CSS}
The category of bisimplicial spaces, that is, the functor category $\sset^{\Delta^{\op}}$, admits a complete Segal space model structure due to Rezk~\cite{Rezk:MHTHT}. 
This is obtained as follows: first one takes the Reedy model structure on $\sset^{\Delta^{\op}}$ (whose weak equivalences are invariant under the action of $C_2$ on $\Delta^{\op}$), and then localizes this with a certain set of maps as in $\sset^{\Delta^{\op}}$.
This set of maps is invariant under the nontrivial action of $C_2$ on $\Delta^{\op}$, thus the fibrant objects are invariant under this action. 
Since all objects are cofibrant, this implies (using Example 17.2.4 and Definition 3.1.4(b) of \cite{Hirschhorn:MCL}) that weak equivalences in this model structure are also preserved by taking opposites, so Corollary~\ref{cor: semidirect} applies. 
Thus there is a model structure on $\sset^{\nabla^{\op}}$ lifted from the complete Segal space model structure.
\end{example}

\begin{example}\label{example: flippy gamma}
If $\Gamma$ is the graphical category of~\cite{HackneyRobertsonYau:IPIWP}, then the category $\sset^{\Gamma^{\op}}$ admits a generalized Reedy model structure, which admits a localization so that fibrant objects are those Reedy fibrant graphical spaces which satisfy a Segal-type condition. One further localization ensures that the underlying simplicial space of any fibrant object is a complete Segal space (see Example~\ref{example: CSS}).

Moreover, $\Gamma$ admits an action of $C_2$ which on objects reverses the directions of graphs. All constructions from the previous paragraph are compatible with this action as they were in Example~\ref{example: CSS}, hence $\sset^{(\Gamma\rtimes C_2)^{\op}}$ has a right-induced model structure.
\end{example}

The considerations of this example do not apply to the dendroidal category $\Omega$, as one cannot reverse the orientation of edges in a rooted tree.
\begin{example}
Consider the category $\Omega_p$ of \emph{planar} rooted trees from~\cite[\S 2.2]{Moerdijk:LDS}.
If $T$ is any planar rooted tree, reflection in the plane along the line spanned by the root edge yields a new planar rooted tree.
This process extends to an action of $C_2$ on $\Omega_p$ as in~\cite[Proposition 6.3.8]{AraGrothGutierrez:OAIOCIO}.
The category of planar dendroidal sets, that is, the presheaf category $\set^{\Omega_p^{\op}}$, admits a cofibrantly generated model category structure with all objects cofibrant~\cite[Theorem 8.2.1]{Moerdijk:LDS}.
One can show that the mirroring isomorphism respects this structure, so that Corollary~\ref{cor: semidirect} guarantees a right-induced model structure on $\set^{(\Omega_p\rtimes C_2)^{\op}}$.
\end{example}

We end this section with a simple alternative description of the category of $\nabla$-presheaves: as a category of simplicial sets with extra structure.
The reader should notice that the argument from \S\ref{section categories with anti involution} for $\catinv$ can be used to recover the model structure from Corollary~\ref{corollary: joyal}; we have elected to take this path in order to demonstrate the use of Theorem~\ref{theorem: semidirect}.

\begin{proposition}
Consider the category $\ssetinv$ whose objects are pairs $(A,\sigma)$, where $A\in \sset$ and $\sigma : A^{\op} \to A$ satisfies $\sigma^{\op}\sigma = \id_{A^{\op}}$. 
A morphism $(A,\sigma) \to (A',\sigma')$ is a map $f: A \to A'$ of simplicial sets so that $\sigma' f^{\op} = f \sigma$.

Then the category $\set^{\nabla^{\op}}$ is equivalent to $\ssetinv$ over $\sset$.
\end{proposition}
\begin{proof}
Let $\sigma$ denote the non-identity element of $C_2$. If $X$ is any $\nabla$-presheaf, then the maps $(\id_{[n]}, \sigma)$ of $\nabla$ induce functions 
$(\id_{[n]}, \sigma)^* : X_n \to X_n$ for all $n$.
Let $\alpha : [m] \to [n]$ be an order-preserving map and let $\mathcal{F}(\alpha) = \tau^n \alpha \tau^m$.
Then
\[ (\id_{[m]}, \sigma) \circ (\mathcal{F}(\alpha), 1) = (\rho_{\sigma}(\mathcal{F}(\alpha)), \sigma) = (\mathcal{F}^2(\alpha), \sigma) = (\alpha, \sigma) = (\alpha, 1) \circ (\id_{[n]}, \sigma) \]
so the diagram
\[ \begin{tikzcd}
X_n \rar{(\id_{[n]}, \sigma)^*} \dar[swap]{(\mathcal{F}(\alpha), 1)^*} & X_n\dar{(\alpha, 1)^*}  \\
X_m \rar{(\id_{[m]}, \sigma)^*} & X_m 
\end{tikzcd} \]
commutes, and we see that the collection $(\id_{[n]}, \sigma)^*$ constitutes a simplicial set map $(\iota^*X)^{\op} \to \iota^*X$. It is clearly anti-involutive, so $\iota^*: \set^{\nabla^{\op}} \to \sset$ factors through the forgetful functor $\ssetinv \to \sset$.
On the other hand, given an object $(A,\sigma) \in \ssetinv$, the simplicial set $A$ admits the structure of a $\nabla$-presheaf by defining $(\alpha,\sigma)^*: A_n \to A_m$ to be 
\[
	A_n \xrightarrow{\sigma} A_n \xrightarrow{\alpha^*} A_m
\]
and $(\alpha, 1)^* : A_n \to A_m$ to be $\alpha^*$.
Thus the forgetful functor $\ssetinv \to \sset$ factors through $\iota^*$, and we get $\ssetinv \cong \set^{\nabla^{\op}}$ over $\sset$.
\end{proof}

\section{Two models for infinity categories with strict anti-involution.}

Consider the adjoint string $(L,F,R)$ between $\simpcatinv$ and $\simpcat$ of Section~\ref{section: simpcat} and the adjoint string $(\iota_*,\iota^*,\iota_!)$ between $\ssetinv \to \sset$ induced by the inclusion $\iota$ of $\Delta$ into the semidirect product $\nabla=\Delta\rtimes C_2$. 
There is a Quillen equivalence~\cite[Theorem 2.2.5.1]{Lurie:HTT} 
\[ \hcr : \sset \rightleftarrows \simpcat : \hcn \]  
between the Joyal model structure and the Bergner model structure.
Our goal for this section is to show that this Quillen equivalence lifts to a Quillen equivalence
\[ \hcrinv : \ssetinv \rightleftarrows \simpcatinv : \hcninv \]
between the right-induced model structures.

\begin{lemma}\label{lemma: w and z exist}
There exist natural isomorphisms $z: (\hcn(-))^{\op}\to \hcn((-)^{\op})$ 
and $w: (\hcr(-))^{\op}\to \hcr((-)^{\op})$ such that 
$z^{-1}=z^{\op}$, $w^{-1}=w^{op}$, 
and $w$ and $z$ are related by the formula\footnote{Here $\varepsilon$ (resp. $\eta$) is the counit (resp. unit) of the adjunction $\hcr \dashv \hcn$.}
\[
w_A^{-1} = \varepsilon_{(\hcr A)^{\op}} \circ \hcr (z_{\hcr A}) \circ \hcr ((\eta_{A})^{\op}).
\]
\end{lemma}
This can be shown by an explicit construction of $z$ (which then determines $w$ by the formula of the lemma), which we omit. 
The existence of natural isomorphisms is well-known (see, for instance,~\cite[\S 1.2.1]{Lurie:HTT}), but to our knowledge no one has previously needed the coherence information of this lemma.

 For a map $\tau:X^{\op}\to X$ in $\simpcat$, define the map $\tau' : \hcn(X)^{\op} \to \hcn(X)$ in $\sset$ to be $\hcn(\tau)\circ z_X$. 
 Similarly, for $\sigma:A^{\op}\to A$ in $\sset$, define $\sigma'$ in $\simpcat$ to be $\hcr(\sigma)\circ w_A$.

The formal properties of being inverse to their opposites immediately yield the following, whose proof we also omit.
 \begin{lemma}
Suppose $\tau:X^{\op}\to X$ is a map in $\simpcat$ and $\sigma:A^{\op}\to A$ is a map in $\sset$.
Then:
 \begin{enumerate}
 \item the map $\tau$ is an anti-involution in $\sset$ if and only if $\tau'$ is an anti-involution in $\simpcat$.
 \item the map $\sigma$ is an anti-involution in $\simpcat$ if and only if $\sigma'$ is an anti-involution in $\sset$.
\end{enumerate}
\end{lemma}
Now we can define a lift of the adjunction between $\sset$ and $\simpcat$ to their anti-involutive versions.
\begin{definition}\label{inv definitions}
	If $(A,\sigma) \in \ssetinv$, define 
	\[ \hcrinv(A,\sigma) \coloneqq (\hcr(A), \sigma') \in \simpcatinv. \]
	If $(X, \tau)$ is an anti-involutive simplicially-enriched category, define
	\[ \hcninv(X,\tau) \coloneqq ( \hcn(X), \tau') \in \ssetinv. \]
\end{definition}
\begin{theorem}
The functor $\hcrinv$ is left adjoint to $\hcninv$.
\end{theorem}
\begin{proof}Let $f:\hcr(A)\to X$ have adjunct morphism $h:A\to\hcn(X)$. 
Let $\tau:X^{\op}\to X$ and $\sigma:A^{\op}\to A$ be any two maps. 
By an argument using the compatibility between $w$ and $z$ of Lemma~\ref{lemma: w and z exist} and naturality of the unit and counit of the $(\hcr,\hcn)$ adjunction, the commutativity of the following two diagrams is equivalent.

	\[ \begin{tikzcd}
	\hcr(A)^{\op} \rar{\sigma'}\dar{f^{\op}} & 
	\hcr{A} \dar{f} 
	&A^{\op} \arrow[r, "\sigma"] \dar{h^{\op}} &   
	A \dar{h}
	\\
	X^{\op} \arrow[r, "\tau"] &  
	X &	
	(\hcn X)^{\op}  \rar{\tau'} & \hcn(X)
	\end{tikzcd} \]
This suffices to show that the lifted functors are adjoint.
\end{proof} 
We now have the following collection of functors.
\[
\begin{tikzcd}[row sep=large]
	\simpcatinv 
	\rar["F" description]
	\dar[shift left, "\hcninv"] & 
	\simpcat 
	\lar[bend right=20, "L" swap] 
	\lar[bend left=20, "R"]
	\dar[shift left, "\hcn"] \\
	\ssetinv \rar["\iota^*" description] \uar[shift left, "\hcrinv"] & \sset \uar[shift left, "\hcr"]
	\lar[bend right=20, "\iota_!" swap] 
	\lar[bend left=20, "\iota_*"] 
\end{tikzcd}
\]
It is immediate from Definition~\ref{inv definitions} that $F\hcrinv =\hcr \iota^*$ and $\iota^*\hcninv = \hcn F$. 
\begin{theorem}\label{theorem: hcn qe lifts}
	The adjoint pair $(\hcrinv, \hcninv)$ is a Quillen equivalence between the model structure on $\simpcatinv$ from~\ref{section: simpcat} and the Joyal model structure on $\ssetinv$ from Corollary~\ref{corollary: joyal}.
\end{theorem}
\begin{proof}
	Suppose that $(A, \sigma)$ is cofibrant in $\ssetinv$, $(X, \tau)$ is fibrant in $\simpcatinv$, and $f: \hcrinv(A,\sigma) \to (X,\tau)$ is adjunct to $h: (A, \sigma) \to \hcninv(X,\tau)$. 
	We must show that $f$ is a weak equivalence if and only if $h$ is.
	Since the model structures are right-induced, $f$ (resp. $h$) is a weak equivalence if, and only if, $f: \hcr(A) \to X$ (resp. $h: A \to \hcn(X)$) is.
	Further, $f: \hcr(A) \to X$ is adjunct to $h: A \to \hcn(X)$.
	It is automatic that $A = \iota^*(A,\sigma)$ is cofibrant since every object of $\sset$ is.
	Since $F$ preserves fibrations and the terminal object, we know that $X = F(X,\tau)$ is fibrant.
	Since $(\hcr, \hcn)$ is a Quillen equivalence between $\simpcat$ and $\sset$, we conclude that
	$f$ is a weak equivalence if and only if $h$ is a weak equivalence.
\end{proof}

As observed by Edoardo Lanari, the proof of Theorem~\ref{theorem: hcn qe lifts} is formal.
\begin{theorem}\label{theorem lanari}
Suppose we have a map of adjunctions (c.f.~\cite[IV.7]{MacLane:CWM}) $(F,F')$ from $(\tilde{A}\dashv \tilde{B})$ to $(A\dashv B)$, that is, a diagram of functors
	\[ \begin{tikzcd}
	\mathcal{N} \dar{F} \rar[shift left]{\tilde{A}} & \mathcal{N}' \lar[shift left]{\tilde{B}} \dar{F'} \\
	\mathcal{M} \rar[shift left]{A} & \mathcal{M}' \lar[shift left]{B} \\
	\end{tikzcd} \]
	where $\tilde{A} \dashv \tilde{B}$ and $A\dashv B$ are adjunctions, $AF = F'\tilde{A}$, $F\tilde{B} = BF'$, and $F' \tilde \varepsilon_f = \varepsilon_{F'f}$.
	Suppose further that all the categories are model categories, $(A,B)$ is a Quillen equivalence, and:
	\begin{itemize}
		\item $F$ and $F'$ create weak equivalences;
		\item $F$ reflects fibrations and preserves cofibrant objects;
		\item $F'$ preserves fibrations and fibrant objects; 
	\end{itemize}
	Then $(\tilde A, \tilde{B})$ is a Quillen equivalence.
\end{theorem}
\begin{proof}
	The functor $\tilde B$ is a right Quillen functor because $BF'$ preserves (acyclic) fibrations and $F$ reflects them.

	Suppose that $c \in \mathcal{N}$ is cofibrant, $f\in \mathcal{N}'$ is fibrant, and $h: c \to \tilde B f$ is adjunct to $h' = \tilde \varepsilon_f \circ \tilde A(h): \tilde A c \to f$.
	Since $F$ creates weak equivalences, $h$ is a weak equivalence if and only if $Fh : Fc \to F \tilde B f = B F' f$ is.
	Since $Fc$ is cofibrant, $F'f$ is fibrant, and $B$ is a right Quillen equivalence, $Fh$ is a weak equivalence if and only if the adjunct $AFc \to F'f$ of $Fh$ is a weak equivalence.
	The compatibility condition between $F$, $F'$, and the adjunctions implies that $F'h'$ is the adjunct of $Fh$.  
	Because $F'$ creates weak equivalences, $F'h'$ is a weak equivalence if and only if $h'$ is a weak equivalence.	
\end{proof}
In particular, suppose that the model structure on $\mathcal{N}$ is lifted from $\mathcal{M}$ along $F$ using Theorem~\ref{theorem: adjoint string, Quillen condition} and that the model structure on $\mathcal{N}'$ is lifted from $\mathcal{M}'$ along $F'$ (which need only be a right adjoint). Then the dashed model-categorical conditions of Theorem~\ref{theorem lanari} are satisfied by assumption.

\begin{example}\label{example CSS vs qcat}
Recall the model structure on $\sset^{\nabla^{\op}} = \set^{(\nabla \times \Delta)^{\op}}$ from Example~\ref{example: CSS} lifted from the complete Segal space model structure.
Consider the pair of adjoint functors $p_1 \dashv i_1$, where $p_1 : \nabla \times \Delta \to \nabla$ is projection and $i_1 = (-) \times [0] : \nabla \to \nabla \times \Delta$.
These induce an adjunction 
\begin{equation}\label{upstairs CSS adj} p_1^* : \set^{\nabla^{\op}} \rightleftarrows \set^{(\nabla \times \Delta)^{\op}} : i_1^* \end{equation}
which maps to the adjunction
\begin{equation}\label{downstairs CSS adj} p_1^* : \set^{\Delta^{\op}} \rightleftarrows \set^{(\Delta \times \Delta)^{\op}} : i_1^* \end{equation}
from~\cite{JoyalTierney:QCSS}. 
One of the main theorems of~\cite{JoyalTierney:QCSS} is that \eqref{downstairs CSS adj} is a Quillen equivalence between the Joyal model structure on the left and Rezk's complete Segal space model structure on the right.
Theorem~\ref{theorem lanari} implies that \eqref{upstairs CSS adj} is a Quillen equivalence as well.
\end{example}

\begin{remark}
In~\cite{HeineLopezAvilaSpitzweck:ICDHMILSM}, the authors define the notion of `infinity category with duality.' 
One considers the $\infty$-category of small $\infty$-categories, which admits an action of $C_2$ by sending an infinity category to its opposite. Then the $\infty$-category of small $\infty$-categories with duality is defined as the homotopy fixed points of this action.
We are curious about how this $\infty$-category compares with the one presented by the model categories in Theorem~\ref{theorem: hcn qe lifts} and Example~\ref{example CSS vs qcat}.
\end{remark}

\appendix
\pdfstringdefDisableCommands{%
  \let\enspace\empty  
}
\section{Proof of Lemma~\ref{lemma: all objects small preserved by exp}}\label{appendix: proof of lemma}

Let $A$ be a functor $\mathcal{C}\to\groundcat$ (which we write using superscripts); by~\cite[Lemma 10.4.6]{Hirschhorn:MCL} there is a cardinal $\kappa'$ so that every $\groundcat$-object in the set $\{ A^c \}_{c\in \ob\,\mathcal{C}}$ is $\kappa'$-small. Let $\kappa$ be a cardinal bigger than both $\kappa'$ and the cardinal of the set $\morphisms\,\mathcal{C}$; we claim that $A$ is $\kappa$-small. Let $\lambda$ be a regular cardinal greater than or equal to $\kappa$ and $Z: \lambda \to \groundcat^{\mathcal{C}}$ be a $\lambda$-sequence. Since colimits in functor categories are pointwise, each $Z^c : \lambda \to \groundcat$ is again a $\lambda$-sequence.
We thus have
\begin{equation}\label{equation: Tdc} T^d_c : \colim_{\beta < \lambda} \groundcat(A^c,Z^d_\beta)\to\groundcat(A^c,\colim_{\beta < \lambda} Z^d_\beta)\end{equation} is a bijection for all $c,d\in \ob\,\mathcal{C}.$
We wish to show that 
\begin{equation}\label{equation: T} T : \colim_{\beta < \lambda} \groundcat^{\mathcal{C}} (A,Z_\beta)\to\groundcat^{\mathcal{C}} (A,\colim_{\beta < \lambda} Z_\beta)\end{equation}
is an bijection as well. We begin with surjectivity.

Suppose that $q : A \to \colim_{\beta < \lambda} Z_\beta$ is in the codomain of $T$.
Let $Q^c_{\alpha(c)} : A^c \to Z_{\alpha(c)}^c$ represent the element $(T^c_c)^{-1}(q^c)$; for any $\beta$ satisfying $\alpha(c) \leq \beta < \lambda$, the map 
\[ Q^c_{\beta} : A^c \xrightarrow{Q^c_{\alpha(c)}} Z_{\alpha(c)}^c \to Z_\beta^c\] also represents $(T^c_c)^{-1}(q^c)$.

Suppose $f: c\to d$ is any map in $\mathcal{C}$ and $\beta$, is less than $\lambda$ but larger than both $\alpha(c)$ and $\alpha(d)$. 
The bottom square and the outer rectangle of the diagram 
\[
\begin{tikzcd}
A^c \rar{A^f} \dar{Q^c_\beta} & A^d \dar{Q^d_\beta} \\
Z_\beta^c \rar{Z_\beta^f} \dar{\eta_\beta^c} & Z_\beta^d \dar{\eta_\beta^d} \\
\colim_\lambda Z^c \rar & \colim_\lambda Z^d
\end{tikzcd}
\]
commute, though the top square may not commute; the vertical composites are $q^c$ and $q^d$.
Thus $Q^d_\beta \circ A^f$ and $Z^f_\beta \circ Q_\beta^c$ both represent the element $(T^d_c)^{-1}(q^d \circ A^f)$;
hence there exists an $\alpha(f)$ bigger than $\alpha(c)$ and $\alpha(d)$ so that $Q^d_{\alpha(f)} \circ A^f = Z^f_\beta \circ Q_{\alpha(f)}^c$.

Since $\lambda$ is a regular cardinal at least as great as $\kappa$, there exists a $\delta$ strictly less than $\lambda$ with $\alpha(f) \leq \delta$ for all $f \in \morphisms\,\mathcal{C}.$
By construction, $Q^d_\delta \circ A^f = Z^f_\delta \circ Q_\delta^c$ for every morphism $f$, so $Q_\delta \in \groundcat^{\mathcal{C}} (A,Z_\delta)$.
Also by construction $T_c^c(\eta_\delta^c \circ Q^c_\delta) = q^c$, hence $T(\eta_\delta \circ Q_\delta) = q$.

Now that we have shown surjectivity, we turn to injectivity.
Suppose $T(p) = T(p')$, where $p$ (resp. $p'$) is represented by $\tilde{p} : A \to Z_{\alpha}$ (resp. $\tilde{p}' : A \to Z_{\alpha'}$); the assumption $T(p) = T(p')$ means $\eta_\alpha \tilde{p} = \eta_{\alpha'} \tilde{p}' : A \to \colim_\lambda Z$.
In particular, by injectivity of $T^c_c$, $\tilde{p}^c$ and $(\tilde{p}')^c$ represent the same element of $\colim_{\beta < \lambda} \groundcat(A^c,Z^c_\beta)$ for every object $c$.
For each $c$, choose some $\alpha(c)$ strictly less than $\lambda$ and at least as great as $\alpha$ and $\alpha'$, and such that 
\begin{equation}\label{equation square} \begin{tikzcd}
	A \rar{\tilde{p}} \dar{\tilde{p}'} & Z_{\alpha} \dar \\
	Z_{\alpha'} \rar & Z_{\beta}
\end{tikzcd}\end{equation}
commutes at $c$ when $\beta$ is greater than or equal to $\alpha(c)$ (and less than $\lambda$).
By the regularity of $\lambda$, there exists $\delta$ strictly less than $\lambda$ and at least as large as all $\alpha(c)$, and the diagram \eqref{equation square} commutes for all $\beta$ at least as large as $\delta$. Thus $p=p'$.

\bibliographystyle{amsalpha}
\raggedright
\bibliography{references-2017}

\affiliationone{Gabriel C.~Drummond-Cole
\\
Center for Geometry and Physics
\\
Institute for Basic Science (IBS)
\\
Pohang 37673, Republic of Korea
\\
\email{gabriel@ibs.re.kr}
}
\affiliationtwo{
Philip Hackney 
\\ 
Department of Mathematics
\\ 
Macquarie University
\\ 
New South Wales, Australia
\\
\email{philip@phck.net} 
}
\affiliationthree{~}
\affiliationfour{
Current address: 
\\ 
Department of Mathematics
\\ 
University of Louisiana at Lafayette
\\ 
Lafayette, LA, United States of America
}
\end{document}